\documentclass[11pt]{article}%
\usepackage{amsfonts}
\usepackage{amssymb}
\usepackage{graphicx}
\usepackage{setspace}
\usepackage[round]{natbib}
\usepackage{amsmath}%
\usepackage{amsthm}%
\usepackage{palatino}
\usepackage{paralist}
\usepackage{enumitem}
\usepackage{hyperref}
\usepackage[top=8pc,bottom=8pc,left=8pc,right=8pc]{geometry}

\DeclareMathOperator\erfc{erfc}

\newcommand{\half}{ {\scriptstyle \frac{1}{2} }  }

\newcommand{\onefour}{ {\scriptstyle \frac{1}{4} }  }

\newcommand{\defeq}{\mathrel{\mathop:}=}
\newtheorem{theorem}{Theorem}

\title{Riemann Hypothesis: a GGC factorisation}
\author{
Nicholas G. Polson\footnote{Nicholas Polson is Professor of Econometrics and Statistics at ChicagoBooth: ngp@chicagobooth.edu.
I would like to thank Lennart Bondesson for many helpful conversions and Jianeng Xu for his comments.}\\
\textit{University of Chicago}
}

\date{October 8, 2018}

\begin{document}
	
\maketitle
\begin{abstract}
	\noindent A GGC (Generalized Gamma Convolution) representation for Riemann's reciprocal $\xi$-function is constructed.
	\vspace{0.5pc}
		
\end{abstract}

\section{Introduction}

Riemann (1859) defines  the $\zeta $-function via the analytic continuation of $ \sum_{n=1}^\infty n^{-s} $ on  the region $ Re(s) > 1 $ and the $\xi$-function by
\begin{equation}
\xi(s) = \half s(s-1) \pi^{-\half s} \Gamma \left ( \half s \right ) \zeta (s)  .
\end{equation}
The Riemann Hypothesis (RH) states that all the non-trivial zeroes of $\zeta(s)$ lie on the critical line $ \text{Re}  (s) = \half $,
or equivalently, those of $ \xi(\half + is) = \xi (\half -i s)$ lie on the real axis.

The $\xi$-function is an entire function of order one and hence admits a Hadamard factorisation. 
Titchmarsh (1974, 2.12.5) shows Hadamard's factorization theorem gives, for all values of $s$, with $b_0 = \half \log ( 4 \pi ) - 1 - \half \gamma $ and $ \xi(0) = - \zeta (0) = \half $,
such that 
\begin{equation}
\xi ( s ) = \xi ( 0)  e^{ b_0 s} \prod_\rho \left ( 1 - \frac{s}{\rho} \right ) e^{ \frac{s}{\rho} } 
\end{equation}
The zeros, $ \rho $,  of $ \xi $ correspond to the non-trivial zeros of $ \zeta$.

The argument to show RH proceeds by showing that 
	it is equivalent to the existence of a generalised gamma convolution (GGC) random variable, denoted by $ H_\half^\xi$, whose Laplace transform expresses the reciprocal
	$ \xi$-function as
	\begin{equation}
	\frac{ \xi \left ( \half \right ) }{ \xi \left ( \half  + \sqrt{s} \right ) } =   E ( \exp ( - s H^\xi_\half )  ) .
	\end{equation}	
This is known as Thorin's condition  (Bondesson, 1992, p.124]).

To see this, first assume that RH is true. Then the zeroes of $ \xi $ are of the form $ \rho = \half \pm i \tau $ as  $ \xi(s)= \xi(1-s)$.
The Hadamard factorisation is then
\begin{equation}\label{xihad1}
\xi(s) =  \xi(0 ) \prod_{ \tau > 0 } \left ( 1 - \frac{s}{ \half + i \tau } \right )\left ( 1 - \frac{s}{ \half - i \tau  } \right )  .
\end{equation}
Now $ \xi(\half) =  \xi(0 ) \prod_{ \tau > 0 } \tau^2 / ( \onefour + \tau^2 ) $ as
\begin{equation}\label{xihad1}
\xi(s) =  \xi(0 ) \prod_{ \tau > 0 } \frac{ ( s - \half )^2 + \tau^2 }{ \onefour + \tau^2 } .
\end{equation}
Hence,  (\ref{xihad1}) shows that the reciprocal $ \xi$-function  satisfies 
\begin{equation}
\frac{ \xi \left ( \half \right ) }{ \xi \left ( \half  + s \right ) }  = \prod_{ \tau > 0 } \frac{\tau^2}{\tau^2 + s^2} .
\end{equation}
Using Frullani's identity, $ \log ( z / z+s^2 ) = \int_0^\infty ( 1 - e^{-s^2 t} ) e^{-tz} dt /t $, write
\begin{align}
\frac{ \xi \left ( \half \right ) }{ \xi \left ( \half  + s \right ) } &  =  \prod_{ \tau > 0 } \frac{\tau^2}{\tau^2 + s^2}
= \exp \left \{ \int_0^\infty \log \left ( \frac{z}{z+s^2} \right ) U(dz) \right \} \\
& =  \exp \left ( - \int_0^\infty ( 1 - e^{- s^2 t} ) g_\half^\xi ( t ) \frac{dt}{t}   \right ) = E ( \exp ( - s^2 H^\xi_\half )  ) .
\end{align}
Here $ g_\half ^\xi( t ) = \int_0^\infty e^{-tz} U_\half ( dz ) $ and $ U_\half (dz) = \sum_{\tau > 0} \delta_{ \tau^2 } ( dz) $ with $ \delta $ a Dirac measure. 

The  GGC random variable $ H^\xi_\half  \stackrel{D}{=} \sum_{ \tau > 0} Y_\tau $ where $ Y_\tau \sim {\rm Exp} ( \tau^2  ) $ satisfies
\begin{equation}
\prod_{\tau > 0} \frac{\tau^2}{\tau^2 + s^2} = E \left ( \exp ( -s^2 H_\half^\xi ) \right ) .
\end{equation}
Conversely, if $ \xi \left ( \half \right ) / \xi \left ( \half  + \sqrt{s} \right ) = E ( \exp ( - s H_\half^\xi )  )  $ then $ \xi( \half + s ) $ has no zeroes. Then
$ \xi (s) $ has no zeroes for $ Re(s) > \half $ and $ \xi(s) = \xi(1-s) $, implies no zeroes for $ Re(s) < \half $ either. 

\subsection{GGC Properties}

The GGC class of probability distributions on $[0,\infty)$ have Laplace transform (LT) which takes the form, with (left-extremity) $ a \geq 0 $, for $ s > 0 $,
$$
\mathbb{E} \left ( e^{-s H} \right ) = \exp \left ( - a s + \int_{(0, \infty)}  \log \left ( z/(z+s) \right ) U(dz) \right )
$$
Here $U(dz)$ a non-negative measure on $(0,\infty)$ (with finite mass on any compact set of $(0, \infty) $)
such that $ \int_{(0,1)}  | \log t | U( d z ) < \infty $ and $ \int_{ (0, \infty) }  z^{-1} U( dz ) < \infty $, see Bondesson (1992).
The $\sigma$-finite measure $ U $ on $ (0 , \infty ) $ is chosen so that the exponent 
\begin{equation}
\phi( s ) = \int_{ ( 0 , \infty )} \log ( 1 + s/z ) U( dz ) = \int_{(0, \infty)\times (0, \infty)}   ( 1 - e^{- sz} ) t^{-1} e^{-tz} U( dz )  < \infty.
\end{equation}
$U$ is often referred to as the Thorin measure and can have infinite mass.
The corresponding L\'evy measure is  $  t^{-1}  \int_{ ( 0 , \infty ) } e^{- t z} U( dz ) $. 

A key property of  the class of GGC distributions is that it is equivalent to the class of generalized convolutions
of mixtures of exponentials (Kent, 1982).  Hence, we can write $ H  \stackrel{D}{=} \sum_{ \gamma > 0} Y_\gamma $ where $ Y_\gamma \sim {\rm Exp} ( \gamma^2 ) $.  This is central to characterizing the zeros of the zeta function and constructing its Hadamard factorization.

\section{Riemann's $\xi$-function and GGC representation}

The following Theorem (Polson, 2017) provides an GGC representation of  
Riemann's reciprocal $\xi$-function.
First, by definition, 
\begin{align}\label{xidefn}	
	\xi(\alpha + s) &  = ( \alpha -1 + s ) \pi^{- \half (\alpha + s)} \Gamma \left ( 1 + \half  ( \alpha + s ) \right ) \zeta(\alpha + s) \\
\xi \left ( \alpha \right )  & =  ( \alpha -1 )  \pi^{- \half \alpha} \Gamma \left ( 1 + \half \alpha \right ) \zeta \left ( \alpha \right ) .
	\end{align}
\begin{theorem}
	Riemann's  reciprocal $ \xi$-function satisfies, for $ \alpha > 1 $ and $ s > 0 $,
\begin{align}\label{xirecip}
	\frac{\xi(\alpha)}{\xi(\alpha+ s)} & = 
	\exp \left ( -  b_\alpha s  + \int_0^\infty ( e^{ - \half s^2 t} -1 ) \frac{ \nu_\alpha^\xi ( t ) }{t} d t  \right ) 
	\end{align}
	where $ b_\alpha = \xi^\prime ( \alpha) / \xi ( \alpha  ) - 1 / (\alpha-1) $.	
Here $\nu_\alpha^\xi(t) = \nu_\alpha^\Gamma(t) + \nu_\alpha^\zeta(t) + \nu_\alpha^0(t)$ with
	\begin{align}
	\nu_\alpha^\Gamma  ( t) &= \frac{1}{\sqrt{2 \pi t}} \int_0^\infty  \frac{ 1 - e^{ - x^2 / 2 t}}{ 1 - e^{-2 x}} e^{- ( \alpha + 2 ) x  }   d x  \\
\nu^\zeta_\alpha ( t) & =  \frac{1}{\sqrt{2\pi t}} \sum_{n \geq 2} \frac{\Lambda(n)}{n^\alpha} \left( 1 - e^{-(\log^2 n)/2t} \right) \\
\nu^0_\alpha (t) &= \half e^{ \half (\alpha-1)^2 t}  \erfc ( (\alpha-1) \sqrt{\half t})\end{align}
	Moreover $ \nu_\alpha^\mu (t ) = \int_0^\infty e^{-tz} U_\alpha^\mu ( d z ) $ is completely monotonic with  
	\begin{equation}\label{xiu} 
	U_\alpha^\mu (z) =\frac{1}{\sqrt{2 \pi}}   \left ( \int_0^\infty   2  \sin^2(x \sqrt{z/2}) e^{-\alpha x} \mu (dx) \right ) \frac{1}{ \sqrt{\pi z}}   \; .
	\end{equation}
\end{theorem}
\begin{proof}	
	From (\ref{xidefn}) with $ \alpha > 1 $ and $ Re(s)> 0 $, 
	\begin{align}\label{xihad}
	\frac{\xi(\alpha+s)  }{\xi(\alpha)}  e^{ -  s \frac{\xi^\prime ( \alpha ) }{\xi ( \alpha  )}  } & =  
	 \left ( 1 + \frac{s}{\alpha-1} \right ) e^{ - \frac{s}{ \alpha -1 }  }  \cdot  \frac{ \Gamma \left (1 +  \half (\alpha+s)  \right ) 
		 }{\Gamma \left ( 1 + \half  \alpha  \right )  }  e^{- s \half \psi ( 1 + \half \alpha ) }
	\cdot \frac{ \zeta( \alpha +s  )  }{ \zeta \left ( \alpha \right ) }   e^{- s \frac{\zeta^\prime}{\zeta} ( \alpha ) }  .
	\end{align}	
	where taking derivatives of $ \log \xi(s) $ at $ s = \alpha $, with $ \psi (s) = \Gamma^\prime (s) / \Gamma (s) $,  gives
	\begin{equation}\label{deriv}
	\frac{\xi^\prime}{\xi} \left ( \alpha \right ) = \frac{1}{\alpha-1} - \half  \log \pi+ \frac{\zeta^\prime}{\zeta} (\alpha) + \half \psi \left ( 1 + \half \alpha \right ) \; . 
	\end{equation}	
Euler's product formula, for $ \alpha > 1  $ and $ Re(s) > 0 $,  now gives 
	\begin{equation}
	\zeta\left ( \alpha + s \right )  = \prod_{ p \;  \text{prime} } \left ( 1 -  p^{-\alpha - s} \right )^{-1} = \prod_{ p \;  \text{prime} } \zeta_p (\alpha + s ) 
	\; \; {\rm where} \; \;  \zeta_p(s) \defeq p^s/ (p^s - 1). 
	\end{equation}
	Using  $ \zeta ( \alpha ) = \prod_p \zeta_p ( \alpha ) $  yields
	\begin{align}\label{zeta1}
	\log \frac{\zeta( \alpha + s) }{\zeta(\alpha)} & = \sum_p  \log \frac{1 - p^{- \alpha}}{1 - p^{-\alpha - s} } 
	= \sum_p \sum_{r=1}^\infty \frac{1}{r} p^{- \alpha r} ( e^{-s r \log p } - 1 ) \\
	& = \int_0^\infty ( e^{-sx} -1 ) e^{-\alpha x} \frac{\mu^\zeta ( d x)}{x} \; 
	{\rm where} \;   \mu^\zeta ( d x) = \sum_p \sum_{r=1}^\infty (\log p)  \delta_{ r \log p} ( d x)  .
	\end{align}
Hence, with $ \Lambda (n) $ the von Mangoldt function, 
		\begin{align}\label{zetaid}
		\frac{\zeta(\alpha+s)} {\zeta(\alpha)} e^{ -\frac{\zeta^\prime ( \alpha ) }{\zeta (\alpha) } s  } & =\exp \left ( \int_0^\infty ( e ^{-sx} +  s x  - 1  ) e^{-\alpha x}  \frac{ \mu^\zeta (dx)}{ x }  \right ) \\
		\frac{\mu^\zeta ( dx )}{x} & = \sum_p \frac{\mu^\zeta_p(dx)}{x} = \sum_{n \geq 2} \frac{ \Lambda (n )}{ \log n} \delta_{ \log n } ( dx ) .
		\end{align}		
	For $ Re(s)> 0 $,  the Gamma function  can be represented as
	\begin{align}	\label{gammaid}
	\frac{ \Gamma \left (1 +  \half (\alpha+s)  \right ) }{\Gamma \left ( 1 + \half  \alpha  \right )  }  e^{- s \half \psi ( 1 + \half \alpha ) } & = 
\exp \left ( \int_0^\infty ( e ^{-sx} +  s x  - 1  ) e^{-\alpha x}  \frac{ \mu^\zeta (dx)}{ x }  \right ) \\
\mu^\Gamma ( d x) & = \frac{dx}{ e^{2x} -1 } .
	\end{align}		 	
The first term on the rhs (\ref{xihad}), for $ \alpha > 1 $,  can be represented as 
	\begin{align}
		\frac{ 1 }{ \alpha -1 + s  } & = \int_0^\infty e^{-sx } e^{- (\alpha-1) x } d x 
		= \int_0^\infty e^{- \half s^2 t }\left \{  \int_0^\infty \frac{ x e^{- \frac{x^2}{2t}} }{ \sqrt{2\pi t^3}} e^{- (\alpha -1 ) x } d x  \right \} dt  \\
		& = \frac{1}{\alpha-1} \exp \left \{  \int_0^\infty ( e^{- \half s^2 t } -1 ) \frac{\nu^0_\alpha (t)}{t} d t \right \}   \label{nuzero}
		\end{align} 
		with completely monotone function 
		\begin{equation}
		\nu^0_\alpha (t) = \half e^{ \half (\alpha-1)^2 t}  \erfc ( (\alpha-1) \sqrt{\half t}) = E( e^{-t Z^0_\alpha} ). 
		\end{equation}  
		Here $ Z^0_\alpha $ has density $ 2(\alpha-1)/ \pi \sqrt{2x}((\alpha-1)^2+2x) $ for $ x> 0$.

For $s>0$, the identity
\begin{equation}\label{levy1}
\frac{e ^{- s x} + s x -1}{x} =
	\int_0^\infty ( 1- e ^{ - \half s^2 t}  ) (1 - e^{-\half x^2/t})  \frac{1}{\sqrt{2\pi t^3}} dt.
\end{equation}	
implies that, for $s>0 $ and $ \mu(dx)$,  
\begin{equation}\label{keyid}
\int_0^\infty ( e^{-sx} + s x -1 ) e^{-\alpha x}  \frac{\mu ( dx )  }{x } = \int_0^\infty ( 1 - e^{ -\half s^2 t}  )  \frac{\nu_\alpha ( t ) }{ t} d t 
\end{equation}
where $ \nu_\alpha (t) $ is the completely monotone function
		\begin{equation}
		\nu_\alpha( t ) = \frac{1}{\sqrt{2 \pi}} \int _0^\infty e^{-tz} \left ( \int_0^\infty   2  \sin^2(x \sqrt{z/2}) e^{-\alpha x} \mu(dx) \right ) \frac{dz}{ \sqrt{\pi z}} .
		\end{equation}
Hence, 
\begin{equation}\label{xihadrep}
\frac{\xi(\alpha)}{\xi( \alpha + s )}    = \exp \left( - b_\alpha  s   +  \int_0^\infty ( e^{   -\half s^2 t}  -1 )  \frac{\nu^\xi_\alpha ( t ) }{ t} d t  \right) 
\end{equation}
where $b_\alpha = \xi^\prime ( \alpha) / \xi ( \alpha  ) - 1/ ( \alpha-1) $.
Here $ \nu^\xi_\alpha ( t ) = \nu^0_\alpha ( t ) + \nu^\Gamma_\alpha (t) + \nu^\zeta_\alpha (t) $ with
\begin{align}
\nu_\alpha^\Gamma  ( t) &= \frac{1}{\sqrt{2 \pi t}} \int_0^\infty  \frac{ 1 - e^{ - x^2 / 2 t}}{ 1 - e^{-2 x}} e^{- ( \alpha + 2 ) x  }   d x  \\
\nu^\zeta_\alpha ( t) & =  \frac{1}{\sqrt{2\pi t}} \sum_{n \geq 2} \frac{\Lambda(n)}{n^\alpha} \left( 1 - e^{-(\log^2 n)/2t} \right) \\
		& =  \frac{1}{\sqrt{2\pi t}} \sum_{p  \; {\rm prime} } \log p \left \{ \sum_{r\geq 1} \frac{1}{p^{\alpha r}} \left( 1 - e^{-(r^2\log^2 p)/2t} \right) \right \}\\
		\nu^0_\alpha (t) &= \half e^{ \half (\alpha-1)^2 t}  \erfc ( (\alpha-1) \sqrt{\half t}).
\end{align}	
\end{proof}	

The LT of a GGC distribution is analytic in the cut plane $\mathbb{C} \setminus (-\infty, 0)$. Hence (\ref{xihadrep}) is also analytic in that cut plane. 
Therefore, by analytic continuation, evaluating (\ref{xihadrep}) at $is + (\half - \alpha)$ and $-is + (\half - \alpha)$ yields
\begin{align}
\frac{\xi(\alpha)}{\xi(\half + is)} &=  \exp\left( -b_\alpha is - b_\alpha (\half - \alpha) + \int_0^\infty (e^{-\half (is + (\half - \alpha))^2 t} -1)  \nu_\alpha^\xi(t) \frac{dt}{t}\right)\\
\frac{\xi(\alpha)}{\xi(\half - is)} &=  \exp\left( b_\alpha is - b_\alpha (\half - \alpha) + \int_0^\infty (e^{-\half (is - (\half - \alpha))^2 t} -1)  \nu_\alpha^\xi(t) \frac{dt}{t}\right). 
\end{align}
Hence,  by symmetry,  $ \xi(\half + is) =\xi(\half - is) $, we have
\begin{equation}
\frac{\xi(\half)}{\xi(\half + is)} = c_\alpha \exp\left(\half \int_0^\infty (e^{-\half (is-(\half-\alpha))^2 t} + e^{-\half (is+(\half-\alpha))^2 t} -2) \nu_\alpha^\xi(t) \frac{dt}{t} \right)
\end{equation}
where $c_\alpha = \exp( - b_\alpha (\half - \alpha))\xi(\half)/\xi(\alpha)$.

\begin{theorem}
There exists a GGC distribution $H_{\half}^\xi$, such that 
\begin{equation}
\frac{\xi(\half)}{\xi(\half - s)} = E[\exp\left(-\half s^2 H_\half^\xi\right)]
\end{equation}
\end{theorem}	
\begin{proof}
As the Laplace transform of GGC is analytic on the cut plane $\mathbb{C} \setminus (-\infty, 0)$ and the RHS of (\ref{xihadrep}). 
As any GGC distribution is a mixture of convolutions of exponentials, we can write, given $\nu_\alpha^\xi(t) = \sum_{\gamma_\alpha} e^{-\gamma_\alpha^2 t}$, Frullani identity $\int_0^\infty (e^{-at} - e^{-bt})dt/t = \log(b/a)$ for $Re(b) > Re(a) >0$, then implies 
\begin{align}
\frac{\xi(\half)}{\xi(\half + is)} &= c_\alpha \exp\left( \half \sum_{\gamma_\alpha} \int_0^\infty (e^{-\half (is-(\half-\alpha))^2 t}-1)e^{-\gamma_\alpha^2 t} \frac{dt}{t}  + \int_0^\infty (e^{-\half (is+(\half-\alpha))^2 t}-1)e^{-\gamma_\alpha^2 t}\frac{dt}{t}\right) \nonumber\\ 
&= c_\alpha \exp\left(\half \sum_{\gamma_\alpha} \log \frac{\gamma_\alpha^4}{(\half s^2 -\gamma_\alpha^2 + \half (\half - \alpha)^2)^2 +  2(\half-\alpha)^2\gamma_\alpha^2}\right)\nonumber\\
&= c_\alpha \exp\left( \sum_{\gamma_\alpha} 
\log \frac{\gamma_\alpha}{\half s^2 - (\gamma_\alpha - \frac{1}{\sqrt{2}}i(\half - \alpha))^2} + \log\frac{\gamma_\alpha}{\half s^2 - (\gamma_\alpha+ \frac{1}{\sqrt{2}}i(\half - \alpha))^2}\right)\nonumber\\
&:= E\left[\exp(\half s^2 H_\half^\xi)\right]
\end{align}
The mgf of a GGC distribution evaluated as $\half s^2$.
\end{proof}	
		
Finally, the Laplace transform, $ E ( \exp ( - s H^\xi_\half )  )   $,  of a GGC distribution, is analytic on the cut plane, namely $ \mathbb{C} \setminus (-\infty, 0) $,
and, in particular, it cannot have any singularities there. By analytic continuation, the same is true of 
$ \xi \left ( \half \right ) / \xi \left ( \half  + \sqrt{s} \right ) $. 
The denominator, $ \xi( \half + \sqrt{s} ) $ cannot have any zeros in the cut plane, and $ \xi( \half + s ) $ has
no zeros for $ Re(s) > 0 $. 
Then $ \xi( s ) $ has no zeroes for $ Re(s) > \half $ and, as  $ \xi(s) = \xi(1-s)$, no zeroes for $ Re(s) < \half $ either. 
Hence all non-trivial zeros of the $\zeta$-function lie on the critical line $ \half + i s $.

\section{{References}}
\noindent Bondesson, L. (1992). \textit{Generalised Gamma Convolutions and Related Classes of Distributions and Densities}. Springer-Verlag, New York.\medskip

\noindent Kent, J. T. (1982). The Spectral Decomposition of a Diffusion Hitting Time. \emph{Annals of Probability}, 10, 207-219.\medskip

\noindent Polson, N. G. (2017). On Hilbert's 8th problem. \emph{arXiv 1708.02653}. \medskip


\noindent Riemann, B. (1859). \"Uber die Anzahl der Primzahlen unter einer gegebenen Gr\"osse. \textit{Monatsberichte der Berliner Akademie}.\medskip

\noindent Titchmarsh, E.C. (1974). \emph{The  Theory of the Riemann Zeta-function}. Oxford University Press.\medskip

\end{document}